\documentclass[12pt]{article}
\usepackage[utf8]{inputenc}
\usepackage{tikz}
\usepackage{amsthm}
\usepackage{amsmath}
\usepackage{amssymb}
\usepackage{amsthm}
\usepackage{amscd}
\usepackage[all]{xy}
\usepackage[margin=1in]{geometry}
\usepackage{comment}
\usepackage{pdfpages}
\usepackage[T1]{fontenc}
\usepackage{amsmath,amsfonts,amssymb,enumerate}
\usepackage{amscd}
\usepackage{graphicx}
\usepackage[T1]{fontenc}
\usepackage{amsmath,amsfonts,amssymb,enumerate}
\usepackage{amsthm} 
\usepackage[english]{babel}

\newtheorem{thm}{Theorem}[section]
\newtheorem{lem}[thm]{Lemma}
\newtheorem{prop}[thm]{Proposition}
\newtheorem{cor}[thm]{Corollary}

\theoremstyle{definition}
\newtheorem{defi}[thm]{Definitions}
\newtheorem{definition}[thm]{Definition}
\newtheorem{remark}[thm]{Remark}
\newtheorem{example}[thm]{Example}

\numberwithin{equation}{section}

\def\min{\mbox{\rm min}}

\def\sdefect{\mbox{\rm sdefect}}

\begin{document}

\title{Asymptotic invariants of ideals with Noetherian symbolic Rees algebra and applications to cover ideals}

\author{Benjamin Drabkin \thanks{ University of Nebraska -- Lincoln, Lincoln Nebraska \textbf{Email address:} benjamin.drabkin@huskers.unl.edu }  \and Lorenzo Guerrieri
\thanks{Universit\`a di Catania, Dipartimento di Matematica e
Informatica, Viale A. Doria, 6, 95125 Catania, Italy \textbf{Email address:} guerrieri@dmi.unict.it }} 

\maketitle






\begin{abstract}
\noindent 
Let $I$ be an ideal whose symbolic Rees algebra is Noetherian. For $m \geq 1$, the $m$-th symbolic defect, $\sdefect(I,m)$, of $I$ is defined to be the minimal number of generators of the module $\frac{I^{(m)}}{I^m}$. We prove that $\sdefect(I,m)$ is eventually quasi-polynomial as a function in $m$.  We compute the symbolic defect explicitly for certain monomial ideals arising from graphs, termed cover ideals.  We go on to give a  formula for the Waldschmidt constant, an asymptotic invariant measuring the growth of the degrees of generators of symbolic powers, for ideals whose symbolic Rees algebra is Noetherian.

\medskip

\end{abstract}

\section{Introduction}
Let $R$ be a commutative Noetherian ring, and let $I\subseteq R$ be an ideal.  The power of $I$ form a descending chain of ideals \[I\supseteq I^2\supseteq I^3\supseteq\cdots\]
which has been the subject of much study.

The generators of the powers of $I$ are not difficult to explicitly compute.  Indeed, $I^n$ is generated by all products of $n$ elements taken among a generating set for $I$.  However, the primary decomposition of the powers of $I$ is much more difficult to compute.  Even in the case that $I=\mathfrak{p}$ is a non-maximal prime ideal, $\mathfrak{p}^n$ accrues new embedded associated primes.  With this in mind, one might be interested in studying in particular those primary components of $I^n$ whose associated primes are also associated primes of $I$.  The \textit{symbolic powers} of $I$ are a construction which captures this information.

\begin{defi}
The $n$-th symbolic power of $I$ is 
\[I^{(n)}:=\bigcap_{p\in Ass(I)}(I^n_p\cap R)\]
\end{defi}
Under this definition, the primary decomposition of $I^{(n)}$ consists of all components of the primary decomposition of $I^n$ whose associated primes are also associated primes of $I$.  In particular, if $\mathfrak{p}$ is prime, then $\mathfrak{p}^{(n)}$ is the $\mathfrak{p}$- primary component of $\mathfrak{p}^n$.  Much like the ordinary powers of $I$, the symbolic powers of $I$ form a descending chain of ideals:
\[I\supseteq I^{(2)}\supseteq I^{(3)}\supseteq\cdots\]

Further motivation for the study of symbolic powers comes from algebraic geometry.  In the case that $R$ is a polynomial ring and $I$ is a radical ideal, the symbolic power has geometric meaning.  Indeed, the results of  Zariski and Nagata in \cite{Zar} and \cite{NAG} show that the $n$-th symbolic power of $I$ consists of all polynomials which vanish to order at most $n-1$ on the variety of $I$.

Studying the relationship between ordinary and symbolic powers of an ideal gives rise to a number of interesting problem.  From the definition it is clear that $I^m\subseteq I^{(m)}$.  The opposite containment, however, does not hold in general.  Much effort has been invested into determining for which values of $r$ the containment $I^{(r)}\subseteq I^m$ holds. An overview of this topic, often called the \textit{containment problem}, can be found in articles like \cite{HOHU}, \cite{ELS}, in the survey paper \cite{SS}, and in \cite{BH}.


While the containment problem is the most-explored line of inquiry into the relationship between symbolic and ordinary powers, there are other avenues to investigate this relationship.  One such method is to study the module $I^{(m)}/I^m$.  This module is relatively unexplored: Herzog in \cite{Herz} studies the module using homological methods when $I$ is a prime ideal of height two in a three-dimensional local Noetherian ring. Arsie and Vatne study the Hilbert function of $I^{(m)}/I^m$ in \cite{AV}, giving examples for ideals of coordinates of planes and set of points in $\mathbb{P}^n$.  More recently, in \cite{GGSV} Galetto, Geramita, Shin, and Van Tuyl have studied this module by defining the \textit{$m$-th symbolic defect},
$$ \sdefect(I,m):=\mu \left( \frac{I^{(m)}}{I^m} \right) ,$$ the number of minimal generators of $I^{(m)}/I^m$. The symbolic defect measures how different $I^{(m)}$ is from $I^m$ in counting the number of generators which must be added to $I^{m}$ in order to make $I^{(m)}$.  For instance, the equality $I^{(m)}=I^m$ is equivalent to $\sdefect(I,m)=0$.  In studying symbolic defect, there are many interesting questions which arise. For instance:
\begin{enumerate}
\item For which ideals is $\sdefect(I,m)$ bounded as a function of $m$? 

\item For which ideals does $\sdefect(I,m)=t$ for given $m,t\in\mathbb{N}$? 

\item How does $\sdefect(I,m)$ grow as $m$ grows to infinity?
\end{enumerate}

 In \cite{GGSV}, the first tow questions are analyzed for the defining ideal of a general set of points in $\mathbb{P}^2$ or a set of points forming a star configuration.
 Interesting results about the symbolic defect of edge ideals of graphs can also be found in the recent work \cite{JKV}.  In this paper, we concern ourselves primarily with the asymptotic behavior of the symbolic defect: with the question of how $\sdefect(I,m)$ grows as $m$ goes to infinity.

Another asymptotic invariant for the study of the containment problem for an ideal $I \subseteq R$ is the \textit{Waldschimdt constant}, defined as $$ \hat{\alpha}(I)=\lim_{m\rightarrow\infty}\frac{\alpha(I^{(m)})}{m},$$ where $ \alpha(I)= \min \left\{ \deg f \, | \, f \in I \right\}.$ This constant was introduced by Waldschmidt in \cite{Wal}. Recently the Waldschimdt constant has been connected to the containment problem by a result of Bocci and Harbourne (\cite{BH} Theorem 1.2). They use it to find a lower bound for the \textit{resurgence} of an ideal $I$, which is defined as $$\rho(I)=\sup \lbrace \frac{m}{r}|I^{(m)}\not\subseteq I^r \rbrace.$$ The Waldschimdt constant of squarefree monomial ideals can be computed as the optimum value of a linear optimization problem as shown in \cite{BCGHJNSVV}.

In this paper, we  study asymptotic invariants pertaining to the symbolic powers of $I$ in the case where $I$ is an ideal with Noetherian symbolic Rees algebra.  A notable class of ideals having this property are monomial ideals \cite{HHT}.  In particular, we focus on the growth of $\sdefect(I,m)$, first by considering the case when $I$ is an ideal in a Noetherial local or graded-local ring, and then specializing to the case in which $I$ is the \textit{cover ideal} of a graph. In addition, we give a formula for the Waldschimdt constant $ \hat{\alpha}(I) $. 

We now describe our main results and the structure of this paper. 
In Section 2, we prove in Theorem \ref{mainhilb} that, when $I$ is an ideal with Noetherian symbolic Rees algebra, $\sdefect(I,m)$ is eventually quasi-polynomial as a function of $m$.


In Section 3, we prove that the Waldschmidt constant of an ideal can be computed by considering only the first few terms of the sequence $\left(\frac{\alpha(I^{(m)})}{m}\right)_{m\in\mathbb{N}}$ as long as $I$ has Notherian symbolic Rees algebra.  In particular, we prove in Theorem \ref{waldgenthm} that the Waldschimdt constant is equal to \[\hat{\alpha}(I)=\min_{m\leq n}\frac{\alpha(I^{(m)})}{m},\] where $n$ is the highest degree of a generator of the symbolic Rees algebra of $I$.



In Section 4, we consider the \textit{cover ideal} $J(G)$ of a graph $G$. Cover ideals are generated by monomials corresponding to vertex covers of graphs and they form an interesting family of squarefree monomial ideals. Their structure has been studied in the recent years due to the relationships between their algebraic and combinatorial properties. We provide all relevant definitions, but for an introductory text about cover and edge ideals we refer to \cite{VT}. In \cite{HHT} it is shown that the symbolic Rees algebra of cover ideals of graphs is generated in degree at most two; this fact allows to easily compute their Waldschimdt constant, applying our result of Section 3. After some preliminary description of vertex covers, we describe 
in Theorem \ref{triangletail} a family of graphs achieving arbitrarily large symbolic defects and successively we recall a known characterization of the graphs having $\sdefect{(J(G),2)=1}$.

In Section 5, we state, in Theorem \ref{sdefformula}, a recursive formula which computes the symbolic defect of cover ideals of a class of graphs having $\sdefect{(J(G),2)=1}$ and we give a criterion to compute
 the degree as a quasi-polynomial of this family of graphs. As applications we compute symbolic defect of complete and cyclic graphs. 


Throughout this paper, for all the standard notations we refer to classical commutative algebra books such as \cite{AM} and \cite{Eis} and for theory of monomial ideals in polynomial rings we refer to the book \cite{HH}.  For an introductory text about cover and edge ideals we refer to \cite{VT}.

\section{Asymptotic behavior of the symbolic defect}
We start by considering the growth of $\sdefect(I,m)$ when $I$ is an ideal (or homogeneous ideal) in a Noetherian local (or graded-local) ring $R$.  

\begin{definition}
\rm A \textit{quasi-polynomial} in $\mathbb{Q}$ is a function $f:\mathbb{N}\rightarrow\mathbb{Q}$ such that for some $n\in\mathbb{N}$ there exist rational polynomials $f_0,\ldots f_{n-1}$ having the property that for all $m\in\mathbb{N}$, $f(m)=f_r(m)$ whenever $m\equiv r$ mod $n$. The value $n$ is called a \textit{quasi-period} of $f$.
\end{definition}

 \begin{definition}
 The \textit{Rees algebra} of an ideal $I$ is defined to be the graded ring $$R(I)=\bigoplus_{i=0}^\infty I^it^i\subseteq R[t] $$ and the \textit{symbolic Rees algebra} of $I$ is defined to be  $$R_s(I)= \bigoplus_{i=0}^{\infty}I^{(i)}t^i\subseteq R[t]$$
 where $t$ is an indeterminate tracking the grading in the graded families of the powers and symbolic powers of $I$, respectively.
 \end{definition}
 
 We recall the useful fact that the  property of a function being eventually quasi-polynomial can be read off its generating function.

\begin{lem}  \rm \cite[4.4.1]{S}
\label{quasilemma}
 \it If a numerical function $\phi:\mathbb{N}\rightarrow\mathbb{Z}$ has generating function \[\sum_{n=0}^\infty \phi(n)z^n=\frac{q(z)}{\prod_{i=1}^s(1-z^{d_i})}\] for some $d_1,d_2,\ldots, d_s\in\mathbb{N}$, and some polynomial $q(z) \in \mathbb{Z}[z]$, then $\phi(n)=Q(n)$ for $n$ sufficiently large, where $Q(n)$ is a quasi-polynomial with quasi-period given by the least common multiple of $d_1,\ldots, d_n$. 
\end{lem}
 
 One of our main results is that $\sdefect(I,m)$ grows quasi-polynomially in $m$ if the symbolic Rees algebra of $I$, $R_s(I)$ is a Noetherian ring.

\begin{thm}
\label{mainhilb}
\label{quasi}
Let $R$ be a Noetherian local or graded-local ring with maximal (or homogeneous maximal) ideal $\mathfrak{m}$ and residue field $R/\mathfrak{m}=\mathbb{K}$.  Let $I$ be an ideal (or homogeneous ideal) of $R$ such that $R_s(I)$ is a Noetherian ring, and let $d_1,\ldots, d_s$ be the degrees of the generators of $R_s(I)$ as an $R$-algebra. Then $\sdefect(I,m)$ is eventually a quasi-polynomial in $m$ with quasi-period $d=\rm{lcm}(d_1,\ldots,d_s)$. 
\end{thm}

\begin{proof}
We first note that for each $m\in\mathbb{N}$, \[0\longrightarrow I^m\longrightarrow I^{(m)}\longrightarrow \dfrac{I^{(m)}}{I^m} \longrightarrow 0\] is an exact sequence of $R$-modules.  Taking direct sums, this gives a short exact sequence of $R(I)$-modules 
\[0\longrightarrow R(I)\longrightarrow R_s(I)\longrightarrow C(I)\longrightarrow 0\] where \[C(I)=\bigoplus_{m=0}^\infty \left( \dfrac{I^{(m)}}{I^m} \right) t^m.\] Tensoring with $\overline{R}(I)=R(I)/\mathfrak{m}R(I)$ gives the exact sequence 
\[0\longrightarrow K\longrightarrow \overline{R}(I)\longrightarrow\overline{R_s}(I)\longrightarrow\overline{C}(I)\longrightarrow 0\] of $\overline{R}(I)$-modules where \[\overline{R_s}(I)=\bigoplus_{m=0}^\infty \left( \dfrac{I^{(m)}}{\mathfrak{m}I^{(m)}} \right) t^m\] and \[\overline{C}(I)=\bigoplus_{m=0}^\infty \left( \dfrac{I^{(m)}}{(I^m+\mathfrak{m}I^{(m)})} \right) t^m.\]  In particular, each strand of the above sequence is of the form \[0\longrightarrow K_m \longrightarrow \dfrac{I^{m}}{\mathfrak{m}I^m}  \longrightarrow \dfrac{I^{(m)}}{\mathfrak{m}I^{(m)}} \longrightarrow \dfrac{I^{(m)}}{(I^m+\mathfrak{m}I^{(m)})}\longrightarrow 0\] where $K_m=(I^m\cap \mathfrak{m}I^{(m)})/\mathfrak{m}I^m$.

Let $h_{\overline{C}(I)}(z)$ be the Hilbert series for $\overline{C}(I)$.  We note that $h_{\overline{C}(I)}(z)=\sum\limits_{m=0}^\infty\sdefect(I,m)z^m$ since, by Nakayama's Lemma, \[\dim_\mathbb{K}\left(\frac{I^{(m)}}{I^m+\mathfrak{m}I^{(m)}}\right)=\mu\left(\frac{I^{(m)}}{I^m}\right).\]  Furthermore, \[h_{\overline{C}(I)}(z)=h_{\overline{R_s}(I)}(z)+h_{K}(z)-h_{\overline{R}(I)}(z)\] where $h_{\overline{R_s}(I)}(z)$, $h_{K}(z)$, and $h_{\overline{R}(I)}(z)$ are the Hilbert series of $\overline{R_s}(I), K,$ and $\overline{R}(I)$ respectively.

Since $\overline{R}(I)$ is a Noetherian $\mathbb{K}$-algebra generated in degree 1 and $K$ is a homogeneous ideal of this ring, $h_{\overline{R}(I)}$ and $h_K$ are rational functions of $z$ with denominators which are powers of $(1-z)$.  Since $R_s(I)$ is a Noetherian ring, it follows that $h_{\overline{R_s}(I)}$ is a rational function with denominator given by $\prod_{i=1}^s(1-z^{d_i})$, where $d_1,d_2,\ldots,d_s$ are the degrees of the generators of $\overline{R_s}(I)$ as a Noetherian $\mathbb{K}$-algebra.  Thus $h_{\overline{C}(I)}(z)$ is a rational function on $z$, and Lemma $\ref{quasilemma}$ show both that $\sdefect(I,m)$ is quasi-polynomial in $m$ and that the quasi-period $d$ is given by $\rm{lcm}(d_1,\ldots,d_s)$.
\end{proof}

We note that, in particular, Theorem \ref{mainhilb} holds for monomial ideals. Sections 4 and 5 of this paper are dedicated to sharpening the results in Theorem \ref{mainhilb} in the case of monomial ideals arising from vertex covers of graphs.

\section{Computing the Waldschmidt constant}

In this section we prove a decomposition for the symbolic powers of an ideal $I$ having Noetherian symbolic Rees algebra and we give an useful application. In order to do this, we use information about the relationships between different symbolic powers of $I$ captured in the symbolic Rees algebra. In particular, when $R_s(I)$ is Noetherian, its maximum  generating degree can be used to understand the structure of $I^{(m)}$, as shown in the following lemma.


\begin{lem}
\label{main1}
Let $R$ be a $\mathbb{N}$-graded Noetherian ring, and let $I\subseteq R$ be a homogeneous ideal such that $R_s(I)$ is generated in degree at most $n$. Then 
\[I^{(m)}=I^m+I^{(2)}I^{(m-2)}+\dots+I^{(n)}I^{(m-n)}\] for all $m>n$.
\end{lem}
\begin{proof}
Since the symbolic powers of $I$ form a graded family of ideals, we have that  \[ I^{(m)}\supseteq I^m+I^{(2)}I^{(m-2)}+\dots+I^{(n)}I^{(m-n)}. \]  
As the symbolic Rees algebra is generated in degree at most $n$, $R_s(I)=R[I,I^{(2)}t^2,\ldots, I^{(n)}t^n]$.   Therefore we have that \[I^{(m)}=\sum_{i_1+2i_2+\ldots+ni_n=m}I^{i_1}(I^{(2)})^{i_2}\cdots(I^{(n)})^{i_n}.\]  Thus \[I^{(m)}\subseteq I^m+I^{(2)}I^{(m-2)}+\dots+I^{(n)}I^{(m-n)}.\]

\end{proof}



\bigskip


Decompositions of symbolic powers of ideals along the lines of Lemma \ref{main1} 
offer a method to computing the Waldschmidt constant of ideals with finitely generated symbolic Rees algebra in terms of the initial degrees of finitely many symbolic powers.

\begin{definition}
\label{defalpha}
\rm Given an ideal $I$ in a $\mathbb{N}$-graded ring $R$, we define $$ \alpha(I)= \min \lbrace \deg f \, | \, f \in I, f\neq 0\rbrace.$$
\end{definition}

\begin{definition} 
\label{wres}
\rm Given an ideal $I$, the \textit{Waldschmidt constant}, $\hat{\alpha}(I)$, is defined by $$ \hat{\alpha}(I)=\lim_{m\rightarrow\infty}\frac{\alpha(I^{(m)})}{m}.$$ 
\end{definition}

For many details about Waldschmidt constant of squarefree monomial ideals see the paper \cite{BCGHJNSVV}. The resurgence of an ideal is another constant related to the containment problem. 

\begin{definition}
\label{defres}
\rm Given an ideal $I$, its the \textit{resurgence}, $\rho(I)$, is defined by $$\rho(I)=\sup \left\{ \frac{m}{r}|I^{(m)}\not\subseteq I^r \right\}.$$
\end{definition}

The resurgence $\rho(I)$ can be bounded below using the Waldschmidt constant.  In particular $$\rho(I)\geq \dfrac{\alpha(I)}{\hat{\alpha}(I)}$$ (see \cite{BH} Theorem 1.2).  

\begin{lem}
\label{waldgeneral}
Let $I$ be a homogeneous ideal such that $R_s(I)$ is generated in degree at most n.  For each $i\in\{1,\dots,n\}$, let $\alpha_i:= \alpha(I^{(i)})$.  Then \[\hat{\alpha}(I)=\lim_{m\rightarrow\infty}\frac{\alpha_1m+(\alpha_2-2\alpha_1)y_2+\dots+(\alpha_n-n\alpha_1)y_n}{m}\] where $y_2,\dots,y_n$ are positive integers minimizing $\alpha_1m+(\alpha_2-2\alpha_1)y_2+\dots+(\alpha_n-n\alpha_1)y_n$ with respect to the constraint $2y_2+3y_3+\dots+ny_n\leq m$.
\end{lem}
\begin{proof}
Since $R_s(I)=R[I,I^{(2)}t^2,\ldots,I^{(n)}t^n]$, we have that \[I^{(m)}=\sum_{2y_2+3y_3+\dots+ny_n\leq m}(I^{(2)})^{y_2}(I^{(3)})^{y_3}\cdots(I^{(n)})^{y_n}I^{m-2y_2-\ldots-ny_n}.\]  
Thus \[\alpha(I^{(m)})=\mbox{min}\left\{\alpha\left((I^{(2)})^{y_2}(I^{(3)})^{y_3}\dots(I^{(n)})^{y_n}I^{m-2y_2-\ldots-ny_n}\right)|2y_2+3y_3+\dots ny_n\leq m \right\},\] and setting $\alpha_i= \alpha(I^{(i)}))$ this gives \[\alpha(I^{(m)})=\min\left\{\alpha_1m+(\alpha_2-2\alpha_1)y_2+\dots+(\alpha_n-n\alpha_1)y_n|2y_2+3y_3+\dots ny_n\leq m\right\}.\]  
\end{proof}

 We proceed to give a formula for the Waldschmidt constant in terms of finitely many symbolic powers.
\begin{thm}
\label{waldgenthm}
Let $I$ be a homogeneous ideal such that $R_s(I)$ is generated in degree at most $n$.  Then \[\hat{\alpha}(I)=\min_{m\leq n}\frac{\alpha(I^{(m)})}{m}.\]
\end{thm}
\begin{proof}
By Lemma \ref{waldgeneral}, $\alpha(I^{(m)})$ is the minimum value of \[\alpha_1m+(\alpha_2-2\alpha_1)y_2+\dots+(\alpha_n-n\alpha_1)y_n\] subject to the condition $2y_2+3y_3+\dots+ny_n\leq m$, where $\alpha_i= \alpha(I^{(i)}))$.  Equivalently, assuming that $n!|m$ and setting $z_i:=iy_i$, we see that $\alpha(I^{(m)})$ is the minimum value of 
\begin{equation}\label{eqmin}\alpha_1m+\frac{(\alpha_2-2\alpha_1)}{2}z_2+\dots+\frac{(\alpha_n-n\alpha_1)}{n}z_n\end{equation} subject to $z_2+z_3+\dots+z_n\leq m$, and for each $i$, $z_i$ is a multiple of $i$.  Let $c\in\{2,\dots,n\}$ such that $\frac{\alpha_c-c\alpha_1}{c}$ is minimal.  Then \begin{equation}\label{eqwal} \alpha_1m+\frac{(\alpha_2-2\alpha_1)}{2}z_2+\dots+\frac{(\alpha_n-n\alpha_1)}{n}z_n\geq m\alpha_1 +\sum_{i=2}^n\frac{\alpha_c-c\alpha_1}{c}z_i.\end{equation} Since $I^i\subseteq I^{(i)}$, we have $\alpha_i-i\alpha_1\leq 0$ for all $i$.  Thus Equation \ref{eqwal} is minimized when $z_2+\cdots+z_n=m$. Hence \[m\alpha_1 +\sum_{i=2}^n\frac{\alpha_c-c\alpha_1}{c}z_i\geq m\alpha_1+m\frac{\alpha_c-c\alpha_1}{c}.\]

Thus, when $n!|m$, (\ref{eqmin}) is minimized at $z_c=m$ and $z_i=0$ for $i\neq c$ with a value of $\frac{m\alpha_c}{c}$.  Therefore $$ \lim_{m\rightarrow\infty}\frac{\alpha(I^{(m)})}{m}=\frac{\alpha_c}{c}.$$
\end{proof}

We note that, in particular, Theorem \ref{waldgenthm} holds for monomial ideals.


\section{Cover ideals and sdefect$\boldsymbol{(I,2)}$}
In Theorem \ref{quasi} we found that $\sdefect(I,m)$ eventually grows quasi-polynomially when $R_S(I)$ is Noetherian.  However, this theorem does not give significant insight into the actual computation of the symbolic defect.  In this section, we turn our attention to a specific class of ideals for which the symbolic defect has combinatorial meaning: cover ideals of graphs.  After defining cover ideals and recalling a few foundational results, we turn our attention to the computation of $\sdefect(I,2)$.  In Theorem \ref{triangletail} we construct a family of cover ideals achieving arbitrarily large $\sdefect(I,2)$, and in Theorem \ref{sdef1}, we describe the class of cover ideals where $\sdefect(I,2)=1$.

In the following $R=\mathbb{K}[x_1,\dots,x_n]$ where $\mathbb{K}$ is a field.


\begin{definition}
\rm Let $G$ be a graph with vertex set $\{x_1,\dots,x_n\}$ and edge set $E$.  The cover ideal of $G$ is defined to be \[J(G):=\bigcap_{\{x_i,x_j\}\in E}(x_i,x_j) \subseteq R.\]
For $m \geq 1$, we say that a monomial $g \in R$ is an $m$-cover of $G$ if for every edge $\{x_i,x_j\}\in E$, there exists one monomial of the form $h_{ij}= x_i^a x_j^b$ with $a+b \geq m$ such that $h_{ij}$ divides $g$.
\end{definition}

The generators of $J(G)$ are the monomials which correspond to the vertex 1-covers of $G$. For any $m\in\mathbb{N}$, we see that \[J(G)^{(m)}=\bigcap_{\{x_i,x_j\}\in E}(x_i,x_j)^m\] is generated by the monomials which correspond to vertex $m$-covers of $G$ and $J(G)^m$ is generated by the monomials which correspond to vertex $m$-covers which decompose into the product of $m$ vertex 1-covers.  We say that a vertex $m$-cover is \textit{minimal} if it is not divisible by any other different vertex $m$-cover.  Thus, in the context of cover ideals, $\sdefect(J(G),m)$ counts the number of minimal vertex $m$-covers of $G$ which cannot be decomposed as a product of $m$ vertex 1-covers. We call such $m$-covers \textit{indecomposable}. 

Herzog, Hibi and Trung proved the next important result:


\begin{thm} [\cite{HHT}, Theorem 5.1]
\label{gendegree2}
Let $G$ be a graph and let $I=J(G)$ be its cover ideal. Then, the symbolic Rees algebra $R_s(I)$ is generated in degree at most $2$.
\end{thm}

As a corollary to this result and Theorem \ref{mainhilb}, we have the following:
\begin{cor} \label{quasiperiodcover}
Let $G$ be a graph and let $I=J(G)$ be its cover ideal.  Then $\sdefect(I,m)$ is eventually quasi-polynomial with quasi-period at most $2$.
\end{cor}




As another consequence of Theorem \ref{gendegree2} and of Theorem \ref{waldgenthm}, we have the following information about the Waldschmidt constant of cover ideals.

\begin{cor}
\label{wald}
Let $I=J(G)$ be the cover ideal of a graph $G$ with $n$ vertices.  Then 
\[\hat{\alpha}(I)=\frac{\alpha(I^{(2)})}{2} \] 
\end{cor}

\begin{proof}
By Theorem \ref{gendegree2}, $R_s(I)$ is generated in degree $2$. Hence, Theorem \ref{waldgenthm} implies that $\hat{\alpha}(I)=\min_{m\leq 2}\frac{\alpha(I^{(m)})}{m}.$ Now the first result follows, since $\frac{\alpha(I^{(2)})}{2} \leq \frac{2\alpha(I)}{2}= \alpha(I).$
\end{proof}

We recall the very well known definition of bipartite graph.

\begin{definition}
\label{bip}
 A graph $G=(V,E)$ is \textit{bipartite} if there is a partition $V= V_1 \cup V_2$ on the vertex set, such that $V_1 \cap V_2 = \emptyset$ and for any edge $ \lbrace x_i, x_j \rbrace$ of $G$, $x_i \in V_1$ and $x_j \in V_2.$
\end{definition}

The following theorem by Dupont and Villareal \cite{DV} characterizes the minimal indecomposable vertex covers of bipartite graphs and the minimal indecomposable vertex 0,1 and 2-covers of non-bipartite graphs. In the latter case, since $R_s(I)$ is generated in degree $2$, these vertex covers generate all the vertex $m$-covers and therefore, in order to understand the symbolic defect of cover ideals, it is important to have a precise overview of indecomposable vertex 2-covers. Given a graph $G=(V,E)$, the set of \textit{neighbors} of a vertex $x_i$ is the set of vertices $x_j$ adjacent to $x_i$, which means $\lbrace x_i, x_j \rbrace$ is an edge of $G$. Given a set of vertices $S \subseteq V$ in $G$, the \textit{induced subgraph} on $S$ is the graph with vertex set $S$, and edge set $\{\{x,y\}\in E|x,y\in S\}$.

\begin{thm}[\cite{DV}, Theorem 1.7]
\label{DupVil}
Let $f=x_1^{a_1}\dots x_n^{a_n}\in \mathbb{K}[x_1,\dots,x_n]$.
\begin{itemize}
    \item[(i)] If $G$ is bipartite, then $f$ is an indecomposable minimal vertex $m$-cover of $G$ if and only if $m=0$ and $f=x_i$ for some $1\leq i\leq n$ or $m=1$ and $f=x_{j_1}\dots x_{j_l}$ is such that $f x_{j_i}^{-1}$ is not a vertex 1-cover of $G$ for every $i={j_1}, \ldots, {j_l}$. 
    For $m \geq 2$, all the vertex $m$-covers of $G$ are divisible by $m$ vertex 1-covers.
    \item[(ii)] If $G$ is non-bipartite, then the minimal indecomposable vertex 0,1 and 2-covers are of the form:
    \begin{itemize}
    \item[(a)](0-covers) $m=0$ and $f=x_i$ for some $i$,
    \item[(b)](1-covers) $m=1$ and $f=x_{j_1}\dots x_{j_l}$ is such that $f x_{j_i}^{-1}$ is not a vertex 1-cover of $G$ for every $i={j_1}, \ldots, {j_l}$, 
    \item[(c)](2-covers) $m=2$ and $f=x_1x_2\cdots x_n$ is the product of all the variables,
    \item[(d)](2-covers) $m=2$ and \[f=x_{i_1}^0\cdots x_{i_s}^0x_{i_{s+1}}^2\cdots x_{i_{s+t}}^2x_{i_{s+t+1}}
    \cdots x_{i_{s+t+u}}\] is such that:
    \begin{itemize}
    \item[(1)]each $x_{i_j}$ is distinct, 
    \item[(2)] $s+t+u=n$,
    \item[(3)]$\{x_{i_{s+1}},\ldots,x_{i_{s+t}}\}$ is the set of neighbors of $\{x_{i_1},\dots,x_{i_s}\}$ in $G$,
    \item[(4)]$g=x_{i_{s+1}}\cdots x_{i_{s+t}}$ is not a vertex cover of $G$, and $u\neq 0$,
    \item[(5)]the induced subgraph on $\{x_{i_{s+t+1}},\ldots, x_{i_{s+t+u}}\}$ has no isolated vertices and is not bipartite. 
    \end{itemize}
    \end{itemize}
\end{itemize}
\end{thm}

An important consequence for the theory of cover ideals of graph is the following result which is also a corollary of work by Sullivant in \cite{Su}: 

\begin{cor}
\label{bipartite}
Let $G$ be a bipartite graph, and let $I=J(G)$.  Then $\sdefect(I,m)=0$ for all $m\in\mathbb{N}$.  
\end{cor}
\begin{proof}
As $G$ is bipartite, by Theorem \ref{DupVil}(i), the graph $G$ does not have indecomposable $m$-covers for $m > 1$.  Thus $\sdefect(I,m)=0$.  
\end{proof}

In contrast to this fact, we show that in general the symbolic defect of a cover ideal can be arbitrarily large. In particular we describe now a family of graphs such that $\sdefect(J(G),2)$ grows as the number of vertices of the graph grows.

Let for $n\in\mathbb{N}$, let $T_n$ be the graph on vertices $\{x_1,x_2,x_3,y_1,\dots,y_n\}$ such that the induced subgraph on $\{x_1,x_2,x_3\}$ is $C_3$ (the odd cycle of lenght three), the induced subgraph on $\{y_1,\dots,y_n\}$ is a path of length $n-1$, and $x_3$ is adjacent to $y_1$.  In the case of $n=3$, this graph is pictured below:

\begin{center}
\begin{tikzpicture}
\begin{scope}[every node/.style={fill=white,circle,thick,draw}]
    \node(A) at (1,1.7) {$x_1$};
    \node(B) at (0,0) {$x_2$};
    \node(C) at (2,0) {$x_3$};
    \node(D) at (4,0) {$y_1$};
    \node(E) at (6,0) {$y_2$};
    \node(F) at (8,0) {$y_3$};
    
\end{scope}
\begin{scope}
            [every edge/.style={draw=black,very thick}]
    \path[-](A)edge node {} (B);
    \path[-](B) edge node {} (C);
    \path[-](A) edge node {} (C);
    \path[-](C) edge node {} (D);
    \path[-](D) edge node {} (E);
    \path[-](E) edge node {} (F);
\end{scope}
\end{tikzpicture}
\end{center}

\begin{thm}
\label{triangletail}
For all $n\geq 5$, $\sdefect(J(T_n),2)=\sdefect(J(T_{n-1}),2)+\mu(J(P_{n-4})^2)$ where $P_i$ is the path of length $i$.
\end{thm}
\begin{proof}
Let $C=x_1^{a_1}x_2^{a_2}x_3^{a_3}y_1^{b_1}y_2^{b_2}\cdots y_n^{b_n}$ be a minimal indecomposable vertex 2-cover of $T_n$. Since every non-bipartite induced subgraph of $T_n$ contains the vertices $x_1, x_2, x_3$, by Theorem \ref{DupVil}(ii.5), we have that $a_1=a_2=a_3=1$.  Thus, $b_1$ is nonzero.  If $b_1=1$, then $b_2>0$.  Consider $$C'=x_1^{a_1}x_2^{a_2}x_3^{a_3}y_1^{b_2}y_2^{b_3}\cdots y_{n-1}^{b_n}.$$  We claim that $C'$ is an indecomposable vertex 2-cover of $T_{n-1}$.  Suppose that $b_i=2$ for some $i$.  Then, as $C$ is a minimal indecomposable vertex 2-cover for $T_n$, we know that either $b_{i-1}=0$ or $b_{i+1}=0$.  Thus, either $y_{i-2}$ or $y_{i}$ has exponent equal to zero in $C'$, 
Thus $C'$ is a vertex 2-cover of $T_{n-1}$. 

On the other hand, if $b_1=2$, then by Theorem \ref{DupVil}(ii) we know that $b_2=0$ and $b_3=2$.  The vertices $y_4,\ldots,y_n$ form the path of length $n-4$, that is $P_{n-4}$, and $y_4^{b_4}\cdots y_n^{b_n}$ is a vertex 2-cover of this path.  Thus, we see that $$\sdefect(J(T_n),2)\geq\sdefect(J(T_{n-1}),2)+\mu(J(P_{n-4})^2).$$

To show the other inequality, let $G=x_1x_2x_3y_1^{a_1}\cdots y_{n-1}^{a_{n-1}}$ be an indecomposable 2-cover for $T_{n-1}$.  Then $C=x_1x_2x_3y_1y_2^{a_1}\cdots y_n^{a_{n-1}}$ is a vertex 2-cover for $T_n$.  Moreover, we note that the set $\{ y \, | \, y^2\mbox{ divides }C \}$ is exactly the set of neighbors of $\{ y \, | \, y\mbox{ does not divide }C \}$.  Thus, again by Theorem \ref{DupVil}, we see that $C$ is an indecomposable 2-cover of $T_n$.

Let $D=y_4^{d_4}\cdots y_n^{d_n}$ be a minimal 2-cover of $P_{n-4}$.  We claim that $H=x_1x_2x_3y_1^2y_2^0y_3^2D$ is an indecomposable 2-cover for $T_n$.  Certainly $H$ is a 2-cover of $T_n$.  As $\{ y \, | \, y^2\mbox{ divides }C \}$ is exactly the set of neighbors of $\{ y \, | \, y\mbox{ does not divide }C \}$ and $H$ is divisible by $x_1x_2x_3$, we see that $H$ is indeed indecomposable by Theorem \ref{DupVil}.
\end{proof}

\medskip

Last theorem shows that a graph having some vertex "very far" from an odd cycle can have very large symbolic defect. Conversely, if every vertex is close to an odd cycle the symbolic defect is going to be reasonably small.
It is interesting to characterize which graphs have cover ideal $I$ with $\sdefect(I,2)=1.$  This characterization   has been done in \cite{HHT} (Proposition 5.3) using the language of symbolic Rees algebras.  Our proof is different from the proof presented in \cite{HHT} and use our terminology introduced above, thus we include it below. 

Recall that a graph is not bipartite if and only if it contains an odd cycle, which means that there is an odd integer $l \geq 3$ and $l$ vertices $x_{i_{1}}, x_{i_{2}}, \ldots, x_{i_l}$ such that for $1 \leq j < l$, $ \lbrace x_{i_{j}}, x_{i_{j+1}} \rbrace$ are edges and $ \lbrace x_{i_{l}},x_{i_{1}} \rbrace$ is an edge.

\begin{thm} 
\label{sdef1}
Let $G$ be a graph with $n$ vertices and let $I=J(G)$. Then $\sdefect(I,2)=1$ if and only if $G$ is non-bipartite and every vertex in $G$ is adjacent to every odd cycle in $G$. \end{thm}
\begin{proof} 
First assume $\sdefect(I,2)=1$. Hence $G$ is not bipartite by Corollary \ref{bipartite} and hence there is at least an odd cycle contained in $G$. Assume by way of contradiction that one vertex $x_i$ of $G$ is not adjacent to the odd cycle and let $x_{j_1}, \ldots, x_{j_c}$ be the vertices of $G$ adjacent to $x_i$. Clearly $F=x_1 x_2 \cdots x_n$ is an indecomposable vertex $2$-cover by Theorem \ref{DupVil}(ii), but in this case, by the same characterization of vertex $2$-covers of graph, we also have that $F x_i^{-1} x_{j_1} \cdots x_{j_c}$ is an indecomposable vertex $2$-cover and hence $\sdefect(I,2) \geq 2$. 

Conversely, assume $G$ is non-bipartite and every vertex in $G$ is adjacent to every odd cycle in $G$ and let $f= x_1^{a_1} \cdots x_n^{a_n} \in I^{(2)} \setminus I^{2}$ be an indecomposable minimal vertex $2$-cover. The set $\lbrace x_i \, : \, a_i=1 \rbrace$ is not bipartite again by Theorem \ref{DupVil}(ii), hence the induced graph on this set contains an odd cycle. Since any vertex of the graph is adjacent to this odd cycle, then $a_i \neq 0$ for every $i=1, \ldots, n$. But this also implies $a_i \neq 2$ for every $i$ since $f$ is a minimal vertex $2$-cover. Hence $f=F$ and $\sdefect(I,2)=1$.
\end{proof}

\begin{remark}
\label{rem2}
\rm From the proof of Theorem \ref{sdef1} it is clear that when $\sdefect(I,2)=1$ where $I=J(G)$, the unique generator of $\frac{I^{(2)}}{I^2}$ is $F=x_1 x_2 \dots x_n$. Hence, for $m \geq 3,$, using the decomposition of $I^{(m)}$ given in Lemma \ref{main1} together with the fact that the symbolic Rees algebra $R_s(I)$ is generated in degree $2$ (see Theorem \ref{gendegree2}), we obtain $$ I^{(m)}=(F)I^{(m-2)}+I^m. $$  
\end{remark}

A consequence of Remark \ref{rem2} is a lower bound on the resurgence of the cover ideal of this kind of graphs. Such a bound is interesting when the degree of $F$ is strictly smaller than $2 \alpha(I)$, where $ \alpha(I)$ is the minimal degree of an element of $I$.

\begin{cor}
\label{resurgencecor}
Let $I=J(G)$ be the cover ideal of a graph $G$ with $n$ vertices such that $\sdefect(I,2)=1$.  Then 
 \[\rho(I) \geq \left\{ \begin{array}{cc} \frac{2\alpha(I)}{n}&\mbox{if }\frac{n}{2}<\alpha(I)\\
1&\mbox{if }\frac{n}{2} \geq \alpha(I). \\
\end{array}\right.  \]
\end{cor}

\begin{proof}
By Remark \ref{rem2}, since the degree of $F=x_1 x_2 \dots x_n$ is $n$ and $F$ is the unique generator of $\frac{I^{(2)}}{I^2}$, we get $\alpha(I^{(2)})= \min\{n, 2 \alpha(I)\}$. Hence the result follows from Definition \ref{defres} and Corollary \ref{wald}.
\end{proof}

\medskip


\section{Cover ideals and the growth of symbolic defect}

We now turn our attention to a more direct computation of the symbolic defect of cover ideals.  Since cover ideals have Noetherian symbolic Rees algebra, we know that their symbolic defects are given by eventually quasi-polynomial functions.  In this section we will directly study these quasi-polynomials for the cover ideals of certain types of graphs. In Theorem \ref{sdefformula} we give an explicit recursive formula for computing this quasi-polynomial for a class of graphs satisfying $\sdefect(J(G),2)=1$. In Proposition \ref{nu} we apply this recursive formula to establish a method for computing the degree of this quasi-polynomial for the same class of graphs.


Using the same notation as before, for a graph $G$ on $n$ vertices, we set $F=x_1 \cdots x_n$ and $I=J(G)= (g_1, \ldots, g_t)$.  Since by Remark \ref{rem2}, $I^{(m)}=(F)I^{(m-2)}+I^m $, the possible elements of $I^{(m)} / I^{m}$ are images of elements of the form $ F^k g_{i_1} \cdots g_{i_s}$ for some integers $k,s$ such that $m=2k+s$. 
It is possible to find an exact formula for the symbolic defect in the case in which all the elements of this form are not in the ordinary power $I^{m}$. For this reason we give the following definition:

\begin{definition}
\label{defstar}
\rm Let $G$ be a graph of $n$ vertices and $I=J(G)$ be its cover ideal. Assume $I=(g_1, \ldots, g_t)$ and $\sdefect(I,2)=1$ and let $F=x_1 \cdots x_n$. We say that $G$ satisfies the 
\it Indecomposability Property \rm if for any integers $k \geq 1$ and $s \geq 0$, the monomial $$F^k g_{i_1} \cdots g_{i_s} \not \in I^{2k+s}.$$
\end{definition}

The following lemma describes some graphs satisfying Indecomposability Property. Anyway, a combinatorial interpretation of this property and a characterization of graphs fulfilling it are actually unknown. 

\begin{lem}
\label{products}
Let $I=J(G)\subseteq R=\mathbb{K}[x_1\dots x_n]$ be the cover ideal of a graph $G$. Assume $\sdefect(I,2)=1$ and let $F=x_1 \cdots x_n$. Then $G$ satisfies the Indecomposability Property if at least one of the following conditions is satisfied:
\begin{enumerate}
\item $I=(g_1, \ldots, g_t)$, $\deg F < 2 \alpha(I)$ and $\deg g_i = \alpha(I)$ for every $i=1, \ldots, c.$ 
\item $I=(g_1, \ldots, g_t, h_1, \ldots, h_s)$ is generated in two different degrees $\deg g_i =\alpha_1 < \alpha_2 = \deg h_l$ for all $i,l$ and $\deg F <  \alpha_1 +\alpha_2$. Moreover there exists $j \in \lbrace 1, \ldots, n \rbrace$ such that the variable $x_{j}$ divides $g_i$ for every $i$ but it does not divide $h_l$ for every $l$.
\item $I=(g_1, \ldots, g_t, h_1, \ldots, h_s)$ is generated in two different degrees $\deg g_i =\alpha_1 < \alpha_2 = \deg h_l$ for all $i,l$. Moreover $\deg F <  \alpha_1 +\alpha_2$, there are $c$ indices $j_1, \ldots, j_c$, such that the variables $x_{j_1}, \ldots, x_{j_c}$ divide $h_l$ for every $l$ and do not divide $g_i$ for every $i$ and $ \alpha_2 -\alpha_1 \leq c$.
\end{enumerate}
\end{lem}

\begin{proof}

\noindent 1) The claim follows since $\deg F^k g_{i_1} \cdots g_{i_s} < (2k+s) \alpha(I) $.\\

\noindent 2) Take an element of the form $$q= F^k  g_{i_1} \cdots g_{i_{s_1}} h_{i_1} \cdots h_{i_{s_2}}.$$ 

It follows that $x_j^{k+s_1}$ divides $q$ but there is not bigger power of $x_j$ dividing $q$. By way of contradiction suppose $q \in I^{2k+s}$. Hence $q$ is divisible by at most $k+s_1$ generators of the form $g_{i}$ and by at least $k+s_2$ generators of the form $h_{l}$. Hence there are two integers $a,b$ such that $a \leq k+ s_1$, $b \geq k+s_2$ and $a+b= 2k+s$ and $q$ is divisible by $a$ generators of the form $g_{i}$ and by $b$ generators of the form $h_{l}$. Then,
$$ \deg q  =k \deg F + s_1 \alpha_1 + s_2 \alpha_2 < (k+ s_1)\alpha_1 + (k+s_2) \alpha_2, $$ but there exists an integer $c \geq 0$ such that $(k+ s_1)\alpha_1 + (k+s_2) \alpha_2 = (a +c)\alpha_1 + (b-c)\alpha_2$ and hence, since $\alpha_1 < \alpha_2$, $$ \deg q < (k+ s_1)\alpha_1 + (k+s_2) \alpha_2 \leq a \alpha_1 + b \alpha_2 \leq \deg q $$ and this is a contradiction for what assumed on $a$ and $b$. \\

\noindent 3) As in the proof of (2), take $$q= F^k  g_{i_1} \cdots g_{i_{s_1}} h_{i_1} \cdots h_{i_{s_2}}.$$ 
Let $X:=x_{k_1} \cdots x_{k_c}$. Thus $X^{k+s_2}$ divides $q$ and no bigger power of $X$ divides it.
Now, assuming by way of contradiction $q \in I^{2k+s}$, we get that $q$ is divisible by at least $k+s_1$ generators of the form $g_{i}$ and by at most $k+s_2$ generators of the form $h_{l}$. As before there are two integers $a,b$ such that $a \leq k+ s_1$, $b \geq k+s_2$, $a+b= 2k+s$ and $m$ is divisible by $a$ generators of the form $g_{i}$ and by $b$ generators of the form $h_{l}$.

Hence there exists an integer $d \geq 0$ such that $k+s_1 = a-d$ and $k+s_2= b+d$ and we can write $\deg q = a \alpha_1 + b \alpha_2 + w$ with $w \geq dc$ since we need to have $X^{k+s_2}$ dividing $q$.
It follows that $$ \deg q  =k \deg F + s_1 \alpha_1 + s_2 \alpha_2 < (k+ s_1)\alpha_1 + (k+s_2) \alpha_2  = a \alpha_1 + b \alpha_2 + d (\alpha_2 -\alpha_1)  \leq \deg q$$ since  $\alpha_2 -\alpha_1 \leq c$ and this is a contradiction.
\end{proof}


We can give examples of graphs satisfying each one of the conditions of Lemma \ref{products}. Recall that a graph is \textit{complete} is for every two distinct vertices $x_i, x_j$, the pair $ \lbrace x_i, x_j \rbrace$ is an edge. We denote the complete graph with $n$ vertices by $K_n$.  A graph is a \textit{cycle} with $n$ vertices $x_1, x_2, \ldots, x_n$ if its edges are of the form $\lbrace x_i, x_{i+1} \rbrace$ modulo $n$. (i.e also $\lbrace x_n, x_{1} \rbrace$ is an edge). We denote the cycle with $n$ vertices by $C_n$. Every cyclic graph with an even number of vertices is bipartite. 

Later complete graphs and the cycles $C_3, C_5, C_7$ will be shown to satisfy condition 1 of Lemma \ref{products}. The following graph satisfies condition 2 of Lemma \ref{products}.

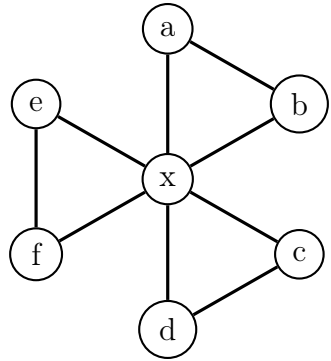
\begin{center}

\begin{tikzpicture}
\begin{scope}[every node/.style={fill=white,circle,thick,draw}]
    \node(A) at (0,2) {a};
    \node(B) at  (1.75,1) {b};
    \node(C) at  (1.75,-1) {c};
    \node(D) at (0,-2) {d};
    \node (E) at (-1.75,1) {e};
    \node (F) at (-1.75,-1) {f};
    \node (G) at (0,0) {x};
\end{scope}
\begin{scope}
            [every edge/.style={draw=black,very thick}]
    \path[-](A)edge node {} (B);
    \path[-](B) edge node {} (G);
    \path[-](A) edge node {} (G);
    \path[-](C) edge node {} (G);
    \path[-](D) edge node {} (G);
    \path[-](F) edge node {} (G);
    \path[-](F) edge node {} (E);
    \path[-](E) edge node {} (G);
    \path[-](D) edge node {} (C);
\end{scope}
\end{tikzpicture}

\end{center}

Similarly, any graph consisting of $3$-cycles all joined at a single vertex satisfies the same condition.  The following graph satisfies condition 3 of Lemma \ref{products}:

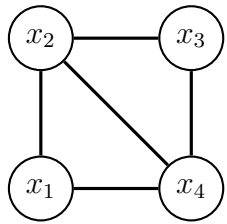
\begin{center}

\begin{tikzpicture}
\begin{scope}[every node/.style={fill=white,circle,thick,draw}]
\node(A) at (0,0) {$x_1$};
\node(B) at (0,2) {$x_2$};
\node(C) at (2,2) {$x_3$};
\node(D) at (2,0) {$x_4$};
\end{scope}
\begin{scope}[every edge/.style={draw=black,very thick}]
            \path[-](A)edge node {} (B);
            \path[-](C)edge node {} (B);
            \path[-](C)edge node {} (D);
            \path[-](A)edge node {} (D);
            \path[-](D)edge node {} (B);
        
\end{scope}
\end{tikzpicture}

\end{center}

\begin{remark} 
\label{rem1}
\rm 
Not every graph satisfies the Indecomposability Property. For instance consider the graph $G=(V,E)$, pictured below. 
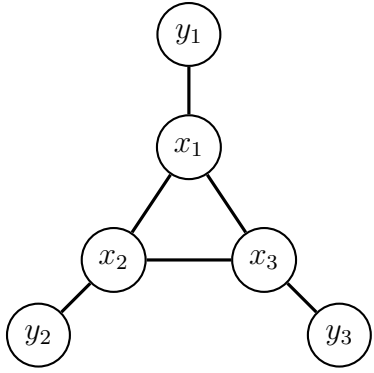
\begin{center}

\begin{tikzpicture}
\begin{scope}[every node/.style={fill=white,circle,thick,draw}]
    \node(A) at (3,1.5) {$x_1$};
    \node(D) at (3,3){$y_1$};
    \node(B) at (2,0) {$x_2$};
    \node(E) at (1,-1) {$y_2$};
    \node(C) at (4,0) {$x_3$};
    \node(F) at (5,-1) {$y_3$};
\end{scope}
\begin{scope}
            [every edge/.style={draw=black,very thick}]
    \path[-](A)edge node {} (B);
    \path[-](A) edge node {} (D);
    \path[-](B) edge node {} (C);
    \path[-](B) edge node {} (E);
    \path[-](A) edge node {} (C);
    \path[-](C) edge node {} (F);
\end{scope}
\end{tikzpicture}

\end{center}

 The cover ideal of this graph is $$ I=(g_1, g_2, g_3, g_4)= (x_1x_2x_3, x_1x_2y_3, x_1x_3y_2, x_2x_3y_1) $$ and we observe that, since $F=x_1x_2x_3y_1y_2y_3$ has degree $6$, the ideal $I$ does not satisfy any of the conditions of Lemma \ref{products}. Moreover we see that $F g_1 = g_2 g_3 g_4 \in I^3$. 
\end{remark}

We now concern ourselves with the computation of the symbolic defect for graphs satisfying the Indecomposability Property. First we need a preliminary Lemma.
\begin{lem} 
\label{uno}
Let $I=J(G)\subseteq R=\mathbb{K}[x_1,\dots, x_n]$ be a cover ideal of a graph, $G$, on $n$ vertices and let $f\in \mathbb{K}[x_1,\dots,x_n]$ be such that $$f\prod_{i\neq a,b}{x_i}\in I^{(m)}$$ for all $a,b\in\{1,\dots,n\}$.  Then $f\in I^{(m)}$.
\end{lem}

\begin{proof} Suppose $f=x_1^{a_1}x_2^{a_2}\cdots x_n^{a_n}\not\in I^{(m)}$.  Then there exists an edge $\{x_s,x_t\}$ such that $a_s+a_t<m$.  But then $f\prod_{i\neq s,t}{x_i}\not\in (x_s,x_t)^m$, and thus $f\prod_{i\neq s,t}{x_i}\not\in I^{(m)}$.
\end{proof}

\begin{thm}
\label{sdefformula}
Let $I=J(G)\subseteq R=\mathbb{K}[x_1\dots x_n]$ where $G$ is a graph such that $\sdefect(I,2)=1$ and let $F=x_1 \cdots x_n.$
Assume that the graph $G$ satisfies the Indecomposability Property. 

Let $\nu(I,m)$ be the number of minimal generators of $I^m$ which are not divisible by $F$. Then for $m \geq 3 $, $$ \sdefect(I,m)= \sdefect(I, m-2) + \nu(I, m-2). $$  
\end{thm}
\begin{proof}
By Remark \ref{rem2} and by assuming $\sdefect(I,2)=1$, we have $I^{(m)}=(F)I^{(m-2)}+I^m $. Then the minimal elements of $I^{(m)} \setminus I^{m}$ are of the form $F^k g_{i_1} \cdots g_{i_s}$ where $g_i$ are minimal generators of $I$ and $m=2k+s$. Conversely, since the graph $G$ satisfies the Indecomposability Property, any such monomial is in $I^{(m)} \setminus I^{m}$. Hence, for $g$ such that $\overline{g}$ is a minimal generator of $\frac{I^{(m-2)}}{I^{m-2}}$, we have that $Fg$ is in $I^{(m)} \setminus I^{m}$. 
 
 The module $I^{(m)}/I^m$ is generated by the images of the minimal generators of $(F)I^{(m-2)}$, which are either generators of $(F)(I^{(m-2)}/I^{m-2})$ or images of generators of $(F)I^{m-2}$. 
 It follows that $\sdefect(I,m)$ is equal to $ \sdefect(I, m-2) $ plus the number of minimal generator of $(F)I^{m-2}$ which are not already multiples of $Fg$ for some $g \in I^{(m-2)} \setminus I^{m-2}$.  

By dividing $F$ on both sides, we have that the last number coincides with the number of minimal generator of $I^{m-2}$ which are not multiple of any element of $I^{(m-2)} \setminus I^{m-2}$ and we will show that this number is $\nu(I, m-2)$. Let $f$ be a minimal generator of $I^{m-2}$. If $f \not \in (F)$, then $f$ is not a multiple of an element of $I^{(m-2)} \setminus I^{m-2}$ since we showed, in Remark \ref{rem2}, that $I^{(m-2)} \setminus I^{m-2} \subseteq (F)$. 

On the other hand, assume $f \in (F)$.  Then $f=Fh$ for some $h\in R$.  Since $Fh\in I^{m-2}$, we know $Fh\in (x_s,x_t)^{m-2}$ for all $x_s, x_t$ where $\{x_s,x_t\}$ is an edge of $G$.  Thus $\frac{F}{x_s x_t}h\in (x_s,x_t)^{m-4}$ and so $h\in I^{(m-4)}$ by Lemma \ref{uno}. 

Then $f \in FI^{(m-4)}$, and so by Remark \ref{rem2}, we get $$f \in (F)I^{(m-4)}= (F)I^{m-4}+(F^2)I^{m-6}+ \ldots+ (F)^{\lfloor \frac{m}{2} \rfloor}J^*$$ where $J^*=R$ for $m$ even and $J^*=I$ for $m$ odd. Notice now that the minimal generators of all the summands of the previous equation are contained in $I^{(m-2)} \setminus I^{m-2}$ (since $F \not \in I^2$). Therefore $f$ is divisible by some element in $I^{(m-2)} \setminus I^{m-2}$. This proves  the formula and the theorem. 
\end{proof}

\begin{remark}
\label{remnew}
\rm Let $I$ be the cover ideal of the graph described in Remark \ref{rem1}. The reader will be able to check that in this case, for $m \geq 3$, $\sdefect(I,m)$ is less than how much is predicted by the formula of Theorem \ref{sdefformula}.
\end{remark}


In the simple case of complete graphs it is possible to apply Theorem \ref{sdefformula} in order to find the explicit value of the symbolic defect. We observe that, as stated in Corollary \ref{quasiperiodcover}, the symbolic defect is a quasipolynomial in $m$ of quasi-period two.

\begin{thm} 
\label{ex1}
Let $K_n$ be the complete graph with $n$ vertices and let $I$ be its cover ideal. Call $F= x_1 x_2 \cdots x_n$. Then \[ \sdefect(I,m) = \left\{ \begin{array}{cc} nk+1 & \mbox{if }m=2k+2 \\
nk & \mbox{if } m=2k+1 \\
 \end{array}\right. \] and hence it grows as a linear quasipolynomial in $m$.
\end{thm}
\begin{proof}
Observe that $I=(g_1, \ldots, g_n)$ where $g_i= F x_i^{-1}$. Clearly by Theorem \ref{sdef1}, $\sdefect(I,2)=1$, and, since $I$ satisfies condition 1 of Lemma \ref{products}, the graph satisfies the Indecomposability Property. 

Since $F$ divides $g_i g_j$ for $i \neq j$, the minimal generators of $I^m$ which are not divisible by $F$ are $g_1^m, g_2^m, \ldots, g_n^m.$ 

It follows that $\nu(I,m)= n$ for every $m \geq 1$ and therefore the formula follows from Theorem \ref{sdefformula}. 
\end{proof}

In other cases of graphs with $\sdefect(J(G),2)=1$ and satisfying the Indecomposability Property, it is possible to establish how fast the symbolic defect is growing (i.e. its degree as a quasipolynomial in $m$) simply looking at the generators of the cover ideal.
To obtain this, we need some further result about the degree in $m$ of the quantity $\nu(I,m)$.
We recall that for an ideal $I \subseteq R=\mathbb{K}[x_1\dots x_n] $, the number of generators of the power $\mu(I^m)$ grows as a polynomial in $m$ for  $m \gg 0$.  This is clear since the fiber cone of $I$ is a standard-graded Noetherian $\mathbb{K}$-algebra and thus its Hilbert function is eventually given by a polynomial (see \cite[Theorem 13.2]{MAT}).  We call the degree of this polynomial the degree of $\mu(I^m)$.

In Proposition \ref{nu} we relate the degree in $m$ of $\nu(I,m)$ with the degree of $\mu(I^m)$. In Proposition \ref{algindip} we recall how to compute $\mu(I^m)$ in the case in which the generators of $I$ are algebraically independent over the base field $k$. We will need this fact when discussing the symbolic defect of cyclic graphs.


\begin{prop} 
\label{nu}
Let $I=J(G)\subseteq R=\mathbb{K}[x_1\dots x_n]$ where $G$ is a graph such that $\sdefect(I,2)=1$ and assume that $G$ satisfies Indecomposability Property. \\
Then for every $i=1, \ldots, n$ we have \begin{equation} \label{nuformula} \mu \left( \left( \dfrac{I+ x_iR}{x_iR} \right)^m \right) \leq \nu(I,m) \leq  \sum_{j=1}^{n} \mu \left( \left( \dfrac{I+ x_jR}{x_jR} \right)^m \right). \end{equation} Moreover, take $i$ such that $\mu(\frac{I+ x_iR}{x_iR}) \geq \mu( \frac{I+ x_jR}{x_jR} ) $ for all $j$ and let $d$ be the degree of $\mu((\frac{I+ x_iR}{x_iR})^m)$ seen as a polynomial in $m$ for $m \gg 0$. Then for $m \gg 0$, $\sdefect(I,m)$ is a quasi-polynomial in $m$ of degree $d+1$.
\end{prop}
\begin{proof}
Recall that from Remark \ref{rem2}, it follows that when the $\sdefect(I,2)=1$, the unique generator of $\frac{I^{(2)}}{I^2}$ is $F=x_1 x_2 \cdots x_n$. Hence $\nu(I,m)$ is the number of minimal generators of $I^m$ not divisible by at least one variable. Since $\mu((\frac{I+ x_iR}{x_iR})^m)$ is the number of minimal generators of $I^m$ not divisible by $x_i$, we can see easily the inequalities in (\ref{nuformula}). 
Now, assuming $$\mu \left( \frac{I+ x_iR}{x_iR} \right) \geq \mu \left( \frac{I+ x_jR}{x_jR} \right) $$ for all $j$, then $\mu((\frac{I+ x_iR}{x_iR})^m)$ and $\sum_{j=1}^{n} \mu((\frac{I+ x_jR}{x_jR})^m)$ have both degree $d$ as polynomials in $m$. 

The sequence $\lbrace \nu(I,m) \rbrace_{m \geq 1}$ is bounded by a polynomial and it is non-decreasing since $F$ is not in the set of minimal generators of $I$. Hence, it admits a (possibly infinite) limit as $m$ grows to infinity. Therefore, dividing both sides of (\ref{nuformula}) by $m^d$ and passing to the limit for $m $ going to infinity, we can see that $\nu(I,m)$ is quasi-polynomial in $m$ of degree $d$. 

For a given $m$, let $k=[\frac{m-2}{2}]$. By Theorem \ref{sdefformula} we have $$\sdefect(I,m)= \sum_{i=0}^{k} \nu(I,2i + c)$$ where $c=0$ if $m$ is even and $c=1$ if $m$ is odd (recall also that $\nu(I,0)=1$ and $\nu(I,1)= \mu(I)$). It follows that $\sdefect(I,m)$ is quasi-polynomial in $m$ of degree one more than the degree of $\nu(I,m)$.
\end{proof}

\begin{prop} 
\label{algindip}
Let $I=(g_1, \ldots, g_s)$ be a squarefree monomial ideal in a polynomial ring over a field $k$. Then: 
\begin{enumerate}
\item The generators of $I$ are algebraically independent over $k$ if the matrix $M_{ij}= \lbrace \frac{\partial g_i}{ \partial x_j} \rbrace$ has maximal rank.
\item If the generators of $I$ are algebraically independent over $k$, then $$\mu(I^m)= \binom{m+s-1}{s-1}$$ is a polynomial in $m$ of degree $s-1.$
\end{enumerate}
\end{prop}
\begin{proof}
We note that 1 follows immediately from Jacobi's criterion \cite{Jac}.
To prove 2, we first note that, since $g_1,\ldots,g_s$ are algebraically independent, $g_1^{a_1}\cdots g_s^{a_s}=g_1^{b_1}\cdots g_s^{b_s}$ if and only if $a_i=b_i$ for $1\leq i\leq s$.  Thus the number of minimal generators of $I^m$ is equal to the number of ways to distribute $m$ objects between $s$ sets.  Therefore $$\mu(I^m)=\binom{m+s-1}{s-1}=\frac{(m+s-1)\cdots(m+1)}{(s-1)!}$$ which is a polynomial in $m$ of degree $s-1$ .
\end{proof}

\begin{example}
We consider the graph in the following picture. Its cover ideal is generated by $$ I= (x_1x_2x_4x_5, x_1x_3x_4, x_1x_3x_5, x_2x_3x_4, x_2x_3x_5)$$ and we observe that $\sdefect(I,2)=1$ and the graph satisfies condition 2 of Lemma \ref{products}. Hence the symbolic defect of this cover ideal can be computed using the formula of Theorem \ref{sdefformula}. Moreover, $$\mu \left( \frac{I+ x_1R}{x_1R} \right) \geq \mu \left( \frac{I+ x_jR}{x_jR} \right) $$ for all $j$, and $J:= \frac{I+ x_1R}{x_1R} = (x_2x_3x_4, x_2x_3x_5)\frac{R}{x_1R}$. Applying Proposition \ref{algindip}, we see that $\mu(J^m)=m+1$, and thus by Proposition \ref{nu} the degree of $\sdefect(I,m)$ as quasipolynomial in $m$ is $2.$
\end{example}

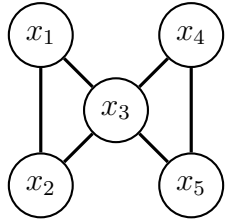
\begin{center}

\begin{tikzpicture}
\begin{scope}[every node/.style={fill=white,circle,thick,draw}]
    \node(A) at (-1,-1) {$x_2$};
    \node(B) at  (-1,1) {$x_1$};
    \node(C) at  (1,1) {$x_4$};
    \node(D) at (1,-1) {$x_5$};
    \node (G) at (0,0) {$x_3$};
\end{scope}
\begin{scope}
            [every edge/.style={draw=black,very thick}]
    \path[-](A)edge node {} (B);
    \path[-](B) edge node {} (G);
    \path[-](A) edge node {} (G);
    \path[-](C) edge node {} (G);
    \path[-](D) edge node {} (G);
    \path[-](D) edge node {} (C);
\end{scope}
\end{tikzpicture}

\end{center}


In Theorem \ref{ncycle}, we are going to use again Propositions \ref{nu} and \ref{algindip} in order to compute an explicit formula of $\sdefect(I,m)$ and its degree in $m$ when $I$ is the cover ideal of a cyclic graph. 

Let $C_n$ be the $n$-cycle and let $I$ be its cover ideal. Since an even cycle is a bipartite graph, it is enough to consider $n=2k+1$ to be an odd number. For every $i=1, \ldots, n$ the monomials $$ g_i= x_i x_{i+2} x_{i+4} \cdots x_{i+n-1} \mbox{ mod } n$$ are $1$-vertex covers of $C_n$ and their degree is minimal among the degrees of the minimal generators of $I$. Indeed, any other generator of $I$, if exists, has degree greater than $\alpha(I)$. We are going to see in Lemma \ref{ncyclelemma}, that if $n \geq 9$, $C_n$ does not satisfy the Indecomposability Property, but later we will be anyway able to prove an alternative recursive formula for its symbolic defect.

\begin{lem} \label{ncyclelemma} Let $n=2k+1$ and let $C_n$ be the $n$-cycle. Let $I$ be the cover ideal of $C_n$ and call $F=x_1 x_2 \cdots x_n$. Assume $G$ is a minimal generators of $I$ such that $\deg G > \alpha(I)$. Then:
\begin{enumerate}
\item There exists a minimal generator $H$ of $I$ and a positive even integer $s<n$ such that $\deg H = \deg G -1$ and $$ G=H x_j x_{j+1}^{-1} x_{j+2} x_{j+3}^{-1} \cdots x_{j+s-1}^{-1} x_{j+s}.$$
\item $FG \in I^3.$ In particular $FG$ is equal to the product of three minimal generators of $I.$
\end{enumerate}
 \end{lem}
 \begin{proof}
\noindent 1) All the generators of $I$ of degree $\alpha(I)$ are of the form $g_i= x_i x_{i+2} x_{i+4} \cdots x_{i+n-1} \mbox{ mod } n$, as mentioned in the paragraph above. Hence, if we assume $\deg G  > \alpha(I)$, we have that there are two different monomials $x_i x_{i+1}$ and $x_j x_{j+1}$ dividing $G$, with $i+1 < j$. We can also assume without loss of generality that, for every $i+1 < h < j$, $x_h$ divides $G$ if and only if $x_{h-1}$ and $x_{h+1}$ do not divide $G$. Clearly, since $G$ is a minimal $1$-cover of $C_n$, $x_{i+2}$ and $x_{j-1}$ do not divide $G$. Thus, we have a sequence of indices $i+1, i+2, i+3, \ldots, j-2, j-1, j$ such that the variables $x_{i+1}, x_{i+3}, \ldots, x_{j-3}, x_{j}$ divide $G$ and all the others in the sequence do not divide $G$. It follows that $j-i$ is an odd number, otherwise we would have two consecutive variables $x_h x_{h+1}$ not dividing $G$ and this is impossible since $\lbrace h, h+1 \rbrace$ is an edge of $C_n$. Hence the monomial $$ H= G x_{i+1}^{-1} x_{i+2} x_{i+3}^{-1} \cdots x_{j-1} x_{j}^{-1}$$ is a well defined $1$-cover of $C_n$ and $\deg H = \deg G -1$. This proves item 1 taking $s= j-i.$
 
 \noindent 2) We claim that $$ FG= F H x_j x_{j+1}^{-1} x_{j+2} x_{j+3}^{-1} \cdots x_{j+s-1}^{-1} x_{j+s} = H g_{j} g_{j+s+1} $$ and hence $FG \in I^3.$ Indeed, this is equivalent to prove $$ g_{j} g_{j+s+1} = F x_j x_{j+1}^{-1} x_{j+2} x_{j+3}^{-1} \cdots x_{j+s-1}^{-1} x_{j+s}. $$ But this follows since, in the case $j$ is odd, $j+s+1$ is even and by definition $$g_j= x_j x_{j+2} x_{j+4} \cdots x_{n} x_2 x_4 \cdots x_{j-3} x_{j-1}$$ and $$g_{j+s+1}= x_{j+s+1} x_{j+s+3} \cdots x_{n-1} x_1 x_3 \cdots x_{j+s-2} x_{j+s}$$ and hence in the products, all the variables with odd index between $j$ and $j+s$ appear twice and all the variables with even index in the same interval do not appear.
 If $j$ is even, the situation is reversed and the result follows in the same way.
  \end{proof}

 \begin{lem} \label{ncyclelemma2}  Let $n=2k+1$ and let $C_n$ be the $n$-cycle. Let $I$ be the cover ideal of $C_n$ and call $F=x_1 x_2 \cdots x_n$. For $m \geq 2$, take $q= F^h f_{1} \cdots f_{s}$ where $f_i$ are minimal generators of $I$ and $m=2h+s$. 
 
 Then, $q \in I^{(m)} \setminus I^{m}$ if and only if $\deg f_i= \alpha(I)$ for every $i$. Moreover, all the monomials generating $ I^{(m)} / I^{m} $ are of this form.
 \end{lem}
 \begin{proof}
 A cyclic graph of odd order has $\sdefect(I,2)=1$ by Theorem \ref{sdef1}. Hence, by Remark \ref{rem2}, we have $I^{(m)}=(F)I^{(m-2)}+I^m $ and the possible minimal generators of $ I^{(m)} / I^{m} $ are all of the form $q= F^h f_{1} \cdots f_{s}$ described in the statement above. 
It is clear that $q \in I^{(m)}$. By item 2 of Lemma \ref{ncyclelemma}, if some $f_i$ has degree greater than $\alpha(I)$  , we note that $Ff_i$ is a product of three minimal generators of $I$, and thus we can rewrite $q$ as a product of $F^{h-1}$ and $s+2$ minimal generators of $I$.  Iterating this process as far as we can, we find two possibilities for $q$, either:

\noindent i)  $q= F^a f_{1} \cdots f_{b}$ with $2a+b=m$ and $\deg f_i= \alpha(I)$ for every $i=1,\ldots, b,$ or

\noindent ii) $q= F f_{1} \cdots f_{b}$ with $b+2=m$ and there exists at least one $f_i$ such that $\deg f_i > \alpha(I).$

In the first case, since $\deg F=n < n+1 = 2\alpha(I)$, we have $$ \deg q= a \deg F + b \alpha(I) < (2a+b) \alpha(I) = \alpha(I^m). $$ and therefore $q \not \in I^m.$

In the second case, assuming $\deg f_1 > \alpha(I)$, we have again by item 2 of Lemma \ref{ncyclelemma}, $q= h_1 h_2 h_3 f_2 \ldots f_b$ with $h_i \in I$, and hence $q \in I^{b+2}= I^m.$ This proves the lemma.
  \end{proof}

 \begin{thm} \label{ncycle}  Let $n=2k+1$ and let $C_n$ be the $n$-cycle. Let $I$ be the cover ideal of $C_n$ and call $F=x_1 x_2 \cdots x_n$. Define $I^{\star} = (g_1, \ldots , g_n)$ where $g_i= x_i x_{i+2} x_{i+4} \cdots x_{i+n-1} \mbox{ mod } n$ and let $\nu(I^{\star},m)$ be the number of minimal generators of $(I^{\star})^m$ which are not divisible by $F$. Then:
 \begin{enumerate}
\item For $m \geq 3$, $\sdefect(I,m)= \sdefect(I,m-2)+ \nu(I^{\star}, m-2).$ 
\item The degree of $ \sdefect(I,m) $ as a quasipolynomial in $m$ is $k$.
\end{enumerate}
 \end{thm}
 \begin{proof}
\noindent 1) In the case $n \leq 7$, it is possible to observe that all the minimal generators of $I$ have the same degree, equal to $k+1$. Hence the graph satisfies the Indecomposability Property, because it satisfies condition 1 of Lemma \ref{products}. In this case, by definition $I=I^{\star}$ and thus the result follows applying Theorem \ref{sdefformula}.
 
 Now suppose that $n \geq 9$. We proceed along the method use to prove Theorem \ref{sdefformula}. Let $m \geq 2$ and let $q \in I^{(m)} \setminus I^{m}$ a minimal generator of $ I^{(m)}/ I^{m}$. By Lemma \ref{ncyclelemma2}, $q= F^h f_{1} \cdots f_{s}$ where $f_i$ are minimal generators and of $I$, $m=2h+s$ and $\deg f_i= \alpha(I)$ for every $i$. Hence for such a monomial $q \in I^{(m-2)} \setminus I^{m-2}$, we have $Fq \in I^{(m)} \setminus I^{m}$, again by Lemma \ref{ncyclelemma2}.
 
Thus we have that, as in the proof of Theorem \ref{sdefformula}, the symbolic defect $\sdefect(I,m)$ is equal to $ \sdefect(I, m-2) $ plus the number of minimal generators of $(F)(I^{\star})^{m-2}$ which are not multiples of $Fg$ for some $g \in I^{(m-2)} \setminus I^{m-2}$.  

Again as in the proof of Theorem \ref{sdefformula}, we have that the last number coincides with the number of minimal generators of $(I^{\star})^{m-2}$ which are not multiples of any element of $I^{(m-2)} \setminus I^{m-2}$ and we can  prove that this number is $\nu(I^{\star}, m-2)$ and this proves the formula in this case, using the same argument. 

 
\noindent 2) Applying the proof of Proposition \ref{nu} to the ideal $I^{\star}$, we see that this fact is equivalent to showing that $\mu((\frac{I^{\star}+ x_iR}{x_iR})^m)$ is a quasipolynomial of degree $k-1$ for some variable $x_i$ which maximizes such degree. The symmetry of the cyclic graphs tell us that $\mu((\frac{I^{\star}+ x_iR}{x_iR})^m)=\mu((\frac{I^{\star}+ x_1R}{x_1R})^m)$ for every $i$. Hence we just consider the ideal $\frac{I^{\star}+ x_1R}{x_1R}=(f_3, f_5, \ldots, f_{n})$. 
 
Notice that $\mu(I^{\star}/x_1)=k$ and we want to conclude the proof applying item 2 of Proposition \ref{algindip}.
In order to do this, by item (i) of Proposition \ref{algindip}, we need to show that the matrix $M= M_{ij}= \lbrace \frac{\partial g_i}{ \partial x_j} \rbrace$ has maximal rank. 

The matrix $M_{ij}$ has $k$ rows and $n-3$ columns. Consider the maximal square submatrix $H$ obtained taking the columns $3, 5, \ldots, n-2, n-1$ of $M$ and let $3 \leq j \leq n-2$ an odd number. Then, by definition of $g_i$ for $i \neq 1$  and odd, we have $$ \frac{\partial g_i}{ \partial x_j}= 0 \Longleftrightarrow x_j \mbox{ does not divide } g_i  \Longleftrightarrow j < i.$$ It is also easy to see that $ \frac{\partial g_i}{ \partial x_{n-1}}=0 $ for $i \neq 1$ odd and $ \frac{\partial g_n}{ \partial x_{n-1}} \neq 0 $. Hence the matrix $H$ is upper triangular with non zero elements on the diagonal and hence the rank of $M$ is maximal.
  \end{proof}

\section*{Acknowledgements}
Most of this work was carried out at the PRAGMATIC workshop at the Universit\'{a} di Catania in the summer of 2017.  We would like to thank the workshop organizers, Alfio Ragusa, Elena Guardo, Francesco Russo, and Giuseppe Zappal\'{a}, as well as the lecturers and collaborators, Brian Harbourne, Adam Van Tuyl, Enrico Carlini, and T\`{a}i Huy H\`{a}.  The first author is partially supported by NSF grant DMS-1601024, EPSCoR grant OIA-1557417, and the University of Nebraska-Lincoln.  Special thanks go to Adam Van Tuyl, Alexandra Seceleanu, T\`{a}i Huy H\`{a}, Jonathan Monta\~{n}o, and Giancarlo Rinaldo for helpful conversations and guidance.


\end{document}